\documentclass{amsart}
\title{A Note on Euclidean Order Types}
\author{Pete L. Clark}
\usepackage{amssymb}

\thanks{Thanks to David Krumm and Robert Varley for providing the inspiration for this work.}

\thanks{$\copyright$ Pete L. Clark, 2012}

\address{Department of Mathematics \\ Boyd Graduate Studies Research Center \\ University 
of Georgia \\ Athens, GA 30602-7403 \\ USA}
\email{pete@math.uga.edu}

\begin{document}
\newtheorem{lemma}{Lemma}
\newtheorem{prop}[lemma]{Proposition}
\newtheorem{cor}[lemma]{Corollary}
\newtheorem{thm}[lemma]{Theorem}
\newtheorem{ques}[lemma]{Question}
\newtheorem{quest}[lemma]{Question}
\newtheorem{conj}[lemma]{Conjecture}
\newtheorem{fact}[lemma]{Fact}
\newtheorem*{mainthm}{Main Theorem}
\newtheorem{obs}[lemma]{Observation}
\newtheorem{hint}{Hint}
\newtheorem{prob}{Problem}

\newcommand{\pp}{\mathfrak{p}}
\newcommand{\mm}{\mathfrak{m}}
\renewcommand{\gg}{\mathfrak{g}}
\newcommand{\DD}{\mathcal{D}}
\newcommand{\F}{\ensuremath{\mathbb F}}
\newcommand{\Fp}{\ensuremath{\F_p}}
\newcommand{\Fl}{\ensuremath{\F_l}}
\newcommand{\Fpbar}{\overline{\Fp}}
\newcommand{\Fq}{\ensuremath{\F_q}}
\newcommand{\PP}{\mathcal{P}}
\newcommand{\PPone}{\mathfrak{p}_1}
\newcommand{\PPtwo}{\mathfrak{p}_2}
\newcommand{\PPonebar}{\overline{\PPone}}
\newcommand{\N}{\ensuremath{\mathbb N}}
\newcommand{\Q}{\ensuremath{\mathbb Q}}
\newcommand{\Qbar}{\overline{\Q}}
\newcommand{\R}{\ensuremath{\mathbb R}}
\newcommand{\Z}{\ensuremath{\mathbb Z}}
\newcommand{\SSS}{\ensuremath{\mathcal{S}}}
\newcommand{\Rn}{\ensuremath{\mathbb R^n}}
\newcommand{\Ri}{\ensuremath{\R^\infty}}
\newcommand{\C}{\ensuremath{\mathbb C}}
\newcommand{\Cn}{\ensuremath{\mathbb C^n}}
\newcommand{\Ci}{\ensuremath{\C^\infty}}
\newcommand{\U}{\ensuremath{\mathcal U}}
\newcommand{\gn}{\ensuremath{\gamma^n}}
\newcommand{\ra}{\ensuremath{\rightarrow}}
\newcommand{\fhat}{\ensuremath{\hat{f}}}
\newcommand{\ghat}{\ensuremath{\hat{g}}}
\newcommand{\hhat}{\ensuremath{\hat{h}}}
\newcommand{\covui}{\ensuremath{\{U_i\}}}
\newcommand{\covvi}{\ensuremath{\{V_i\}}}
\newcommand{\covwi}{\ensuremath{\{W_i\}}}
\newcommand{\Gt}{\ensuremath{\tilde{G}}}
\newcommand{\gt}{\ensuremath{\tilde{\gamma}}}
\newcommand{\Gtn}{\ensuremath{\tilde{G_n}}}
\newcommand{\gtn}{\ensuremath{\tilde{\gamma_n}}}
\newcommand{\gnt}{\ensuremath{\gtn}}
\newcommand{\Gnt}{\ensuremath{\Gtn}}
\newcommand{\Cpi}{\ensuremath{\C P^\infty}}
\newcommand{\Cpn}{\ensuremath{\C P^n}}
\newcommand{\lla}{\ensuremath{\longleftarrow}}
\newcommand{\lra}{\ensuremath{\longrightarrow}}
\newcommand{\Rno}{\ensuremath{\Rn_0}}
\newcommand{\dlra}{\ensuremath{\stackrel{\delta}{\lra}}}
\newcommand{\pii}{\ensuremath{\pi^{-1}}}
\newcommand{\la}{\ensuremath{\leftarrow}}
\newcommand{\gonem}{\ensuremath{\gamma_1^m}}
\newcommand{\gtwon}{\ensuremath{\gamma_2^n}}
\newcommand{\omegabar}{\ensuremath{\overline{\omega}}}
\newcommand{\dlim}{\underset{\lra}{\lim}}
\newcommand{\ilim}{\operatorname{\underset{\lla}{\lim}}}
\newcommand{\Hom}{\operatorname{Hom}}
\newcommand{\Ext}{\operatorname{Ext}}
\newcommand{\Part}{\operatorname{Part}}
\newcommand{\Ker}{\operatorname{Ker}}
\newcommand{\im}{\operatorname{im}}
\newcommand{\ord}{\operatorname{ord}}
\newcommand{\unr}{\operatorname{unr}}
\newcommand{\B}{\ensuremath{\mathcal B}}
\newcommand{\Ocr}{\ensuremath{\Omega_*}}
\newcommand{\Rcr}{\ensuremath{\Ocr \otimes \Q}}
\newcommand{\Cptwok}{\ensuremath{\C P^{2k}}}
\newcommand{\CC}{\ensuremath{\mathcal C}}
\newcommand{\gtkp}{\ensuremath{\tilde{\gamma^k_p}}}
\newcommand{\gtkn}{\ensuremath{\tilde{\gamma^k_m}}}
\newcommand{\QQ}{\ensuremath{\mathcal Q}}
\newcommand{\I}{\ensuremath{\mathcal I}}
\newcommand{\sbar}{\ensuremath{\overline{s}}}
\newcommand{\Kn}{\ensuremath{\overline{K_n}^\times}}
\newcommand{\tame}{\operatorname{tame}}
\newcommand{\Qpt}{\ensuremath{\Q_p^{\tame}}}
\newcommand{\Qpu}{\ensuremath{\Q_p^{\unr}}}
\newcommand{\scrT}{\ensuremath{\mathfrak{T}}}
\newcommand{\That}{\ensuremath{\hat{\mathfrak{T}}}}
\newcommand{\Gal}{\operatorname{Gal}}
\newcommand{\Aut}{\operatorname{Aut}}
\newcommand{\tors}{\operatorname{tors}}
\newcommand{\Zhat}{\hat{\Z}}
\newcommand{\linf}{\ensuremath{l_\infty}}
\newcommand{\Lie}{\operatorname{Lie}}
\newcommand{\GL}{\operatorname{GL}}
\newcommand{\End}{\operatorname{End}}
\newcommand{\aone}{\ensuremath{(a_1,\ldots,a_k)}}
\newcommand{\raone}{\ensuremath{r(a_1,\ldots,a_k,N)}}
\newcommand{\rtwoplus}{\ensuremath{\R^{2  +}}}
\newcommand{\rkplus}{\ensuremath{\R^{k +}}}
\newcommand{\length}{\operatorname{length}}
\newcommand{\Vol}{\operatorname{Vol}}
\newcommand{\cross}{\operatorname{cross}}
\newcommand{\GoN}{\Gamma_0(N)}
\newcommand{\GeN}{\Gamma_1(N)}
\newcommand{\GAG}{\Gamma \alpha \Gamma}
\newcommand{\GBG}{\Gamma \beta \Gamma}
\newcommand{\HGD}{H(\Gamma,\Delta)}
\newcommand{\Ga}{\mathbb{G}_a}
\newcommand{\Div}{\operatorname{Div}}
\newcommand{\Divo}{\Div_0}
\newcommand{\Hstar}{\cal{H}^*}
\newcommand{\txon}{\tilde{X}_0(N)}
\newcommand{\sep}{\operatorname{sep}}
\newcommand{\notp}{\not{p}}
\newcommand{\Aonek}{\mathbb{A}^1/k}
\newcommand{\Wa}{W_a/\mathbb{F}_p}
\newcommand{\Spec}{\operatorname{Spec}}
\newcommand{\MaxSpec}{\operatorname{MaxSpec}}
\newcommand{\underint}{\underline{\int}}
\newcommand{\overint}{\overline{\int}}

\newcommand{\abcd}{\left[ \begin{array}{cc}
a & b \\
c & d
\end{array} \right]}

\newcommand{\abod}{\left[ \begin{array}{cc}
a & b \\
0 & d
\end{array} \right]}

\newcommand{\unipmatrix}{\left[ \begin{array}{cc}
1 & b \\
0 & 1
\end{array} \right]}

\newcommand{\matrixeoop}{\left[ \begin{array}{cc}
1 & 0 \\
0 & p
\end{array} \right]}

\newcommand{\w}{\omega}
\newcommand{\Qpi}{\ensuremath{\Q(\pi)}}
\newcommand{\Qpin}{\Q(\pi^n)}
\newcommand{\pibar}{\overline{\pi}}
\newcommand{\pbar}{\overline{p}}
\newcommand{\lcm}{\operatorname{lcm}}
\newcommand{\trace}{\operatorname{trace}}
\newcommand{\OKv}{\mathcal{O}_{K_v}}
\newcommand{\Abarv}{\tilde{A}_v}
\newcommand{\kbar}{\overline{k}}
\newcommand{\Kbar}{\overline{K}}
\newcommand{\pl}{\rho_l}
\newcommand{\plt}{\tilde{\pl}}
\newcommand{\plo}{\pl^0}
\newcommand{\Du}{\underline{D}}
\newcommand{\A}{\mathbb{A}}
\newcommand{\D}{\underline{D}}
\newcommand{\op}{\operatorname{op}}
\newcommand{\Glt}{\tilde{G_l}}
\newcommand{\gl}{\mathfrak{g}_l}
\newcommand{\gltwo}{\mathfrak{gl}_2}
\newcommand{\sltwo}{\mathfrak{sl}_2}
\newcommand{\h}{\mathfrak{h}}
\newcommand{\tA}{\tilde{A}}
\newcommand{\sss}{\operatorname{ss}}
\newcommand{\X}{\Chi}
\newcommand{\ecyc}{\epsilon_{\operatorname{cyc}}}
\newcommand{\hatAl}{\hat{A}[l]}
\newcommand{\sA}{\mathcal{A}}
\newcommand{\sAt}{\overline{\sA}}
\newcommand{\OO}{\mathcal{O}}
\newcommand{\OOB}{\OO_B}
\newcommand{\Flbar}{\overline{\F_l}}
\newcommand{\Vbt}{\widetilde{V_B}}
\newcommand{\XX}{\mathcal{X}}
\newcommand{\GbN}{\Gamma_\bullet(N)}
\newcommand{\Gm}{\mathbb{G}_m}
\newcommand{\Pic}{\operatorname{Pic}}
\newcommand{\FPic}{\textbf{Pic}}
\newcommand{\solv}{\operatorname{solv}}
\newcommand{\Hplus}{\mathcal{H}^+}
\newcommand{\Hminus}{\mathcal{H}^-}
\newcommand{\HH}{\mathcal{H}}
\newcommand{\Alb}{\operatorname{Alb}}
\newcommand{\FAlb}{\mathbf{Alb}}
\newcommand{\gk}{\mathfrak{g}_k}
\newcommand{\car}{\operatorname{char}}
\newcommand{\Br}{\operatorname{Br}}
\newcommand{\gK}{\mathfrak{g}_K}
\newcommand{\coker}{\operatorname{coker}}
\newcommand{\red}{\operatorname{red}}
\newcommand{\CAY}{\operatorname{Cay}}
\newcommand{\ns}{\operatorname{ns}}
\newcommand{\xx}{\mathbf{x}}
\newcommand{\yy}{\mathbf{y}}
\newcommand{\E}{\mathbb{E}}
\newcommand{\rad}{\operatorname{rad}}
\newcommand{\Top}{\operatorname{Top}}
\newcommand{\Sym}{\operatorname{Sym}}
\newcommand{\Cl}{\operatorname{Cl}}
\newcommand{\len}{\operatorname{len}}
\newcommand{\Ord}{\mathbf{Ord}}
\newcommand{\Iso}{\operatorname{Iso}}
\newcommand{\Euc}{\operatorname{Euc}}
\newcommand{\overphi}{\overline{\varphi}}
\newcommand{\underphi}{\underline{\varphi}}
\newcommand{\Prin}{\operatorname{Prin}}


\begin{abstract}
 Euclidean functions with values in an 
arbitrary well-ordered set were first considered in a 1949 work of Motzkin and studied in more detail in work of Fletcher, Samuel and Nagata in the 1970's and 1980's.  Here these results are revisited, simplified, and extended.  The two main themes are (i) consideration of $\Ord$-valued functions on an Artinian poset and (ii) use of ordinal arithmetic, including the Hessenberg-Brookfield ordinal sum.  In particular, to any Euclidean ring we associate an ordinal 
invariant, its \textbf{Euclidean order type}, and we initiate a study of this invariant.  The main new result gives upper and lower bounds on the Euclidean order type of a finite product of Euclidean rings in terms of the Euclidean order types of the factor rings.  
\end{abstract}

\maketitle
\noindent
Throughout, ``a ring'' means a commutative ring with multiplicative 
identity.  We denote by $R^{\bullet}$ the set $R \setminus \{0\}$ and by $R^{\times}$ the group of units of $R$.  We denote by $\N$ the natural numbers, including $0$.  When we are thinking of $\N$ as the least infinite ordinal, we denote it by $\omega$.  If $S$ is any class equipped with a ``zero element'' -- here, either the least element of an ordered class or the identity element of a monoid -- then we put $S^{\bullet} = S \setminus \{0\}$.   \\ \indent
By an \textbf{ordered set} $X$ we mean a pair $(X,\leq)$ with $X$ a set and 
$\leq$ a reflexive, anti-symmetric, transitive relation on $X$ (i.e., what 
is often called a \emph{partial ordering}).  \\ \indent
We will encounter some ordinal arithmetic, and it is important to remember that 
the ``ordinary'' sum and product of transfinite ordinal numbers need not be 
commutative.  The literature seems to agree that for $\alpha,\beta \in \Ord$, 
$\alpha + \beta$ should be the order type of a copy of $\beta$ placed \emph{above} a copy of $\alpha$, so that $\omega + 1 > \omega$, $1 + \omega = \omega$.   However, both conventions on $\alpha \beta$ seem to be in use.  We take the one in which $2\omega = \omega + \omega$, 
\emph{not} the one in which $2\omega = 2 + 2 + \ldots = \omega$.

\section{Ordered Classes, Isotone Maps, and Length Functions}

\subsection{Isotone Maps and Artinian Ordered Classes}
\textbf{} \\ \\ \noindent
For a set $X$, let $\Ord^X$ denote the class of all maps $f: X \ra \Ord$, 
ordered by $f \leq g \iff \forall x \in X, \ f(x) \leq g(x)$.\footnote{We hasten to reassure the reader that this is the limit of 
our set-theoretic ambitiousness: we will never consider the collection of 
all maps between two proper classes!} Note that every nonempty subclass $\mathcal{C} = \{f_c\}$ has an infimum in $\Ord^X$: that is, there is a largest element 
$f \in \Ord^X$ with the property that $f \leq f_c$ for all $f_c \in \mathcal{C}$.  Indeed, we may take $f_{\mathcal{C}}(x) = \min_c f_c(x)$. 
\\ \indent
An ordered class $\mathcal{C}$ is \textbf{downward small} if 
for all $x \in \mathcal{C}$, $\{y \in X \ | \ y \leq x\}$ is a set.  For any set $X$, $\Ord^X$ is downward small.
\\ \\
Let $X$ and $Y$ be ordered classes.  We denote by $X \times Y$ the Cartesian product endowed with the ordering $(x_1,y_1) \leq (x_2,y_2) \iff x_1 \leq x_2$ and $y_1 \leq y_2$.  A map $f: X \ra Y$ is \textbf{weakly isotone} (resp. \textbf{isotone}) if $x_1 \leq x_2 \in X \implies f(x_1) \leq f(x_2)$ (resp. $x_1 < x_2 \implies f(x_1) < f(x_2)$).  The composite of (weakly) isotone maps is (weakly) isotone.  
\\ \\
An ordered class $X$ is \textbf{Noetherian} (resp. \textbf{Artinian}) if there is no isotone map $f: \Z^+ \ra X$ (resp. $f: \Z^- \ra X$).  Thus a well-ordered class is precisely a linearly ordered Artinian class.    
\\ \\
For an ordered \emph{set} $X$, we define $\Iso(X) \subset \Ord^X$ to be the subclass of isotone maps, with the induced partial ordering.  
\begin{lemma}
\label{0.0.1}
Let $X$ be an ordered set.  Then every nonempty subclass of $\Iso(X)$ has an infimum in $\Iso(X)$.  
\end{lemma}
\begin{proof} Let $\mathcal{C} = \{f_c\}$ be a nonempty subclass of $\Iso(X)$, and let $f$ be the infimum in $\Ord^X$; it suffices to show that 
$f$ is isotone.  If $x < y$ in $X$, then $f_i(x) < f_i(y)$ for all $i \in I$, so $f(x) = \min_{i \in I} f_i(x) < \min_{i \in I} (f_i(x) + 1) \leq \min_{i \in I} 
f_i(y) = f(y)$.
\end{proof}

\begin{thm}
\label{0.0.2}
For a downward small ordered class $X$, TFAE: \\
(i) There is an isotone map $f: X \ra \Ord$.  \\
(ii) There is an Artinian ordered set $Y$ and an isotone map $f: X \ra Y$.  \\
(iii) $X$ is Artinian. 
\end{thm}
\begin{proof}
(i) $\implies$ (ii): If $f: X \ra \Ord(X)$ is an isotone map, then 
$f: X \ra f(X)$ is an isotone map with codomain an Artinian ordered set.  \\
(ii) $\implies$ (iii): Suppose not: then there is an isotone map $\iota: \Z^- \ra X$ is an isotone map.  But then $f \circ \iota: \Z^- \ra Y$ is isotone, so $Y$ is not Artinian.  \\
(iii) $\implies$ (i) \cite[Prop. 4]{Nagata85}: We will construct $\lambda_X \in \Iso(X)$ by a transfinite process.  For $\alpha \in \Ord$, at the $\alpha$th stage we assign some subset $X_{\alpha} \subset X$ the value $\alpha$.   Let $X_0$ be the set of minimal elements of $X$.  Having defined $X_{\beta}$ for all $\beta < \alpha$, we assign the value $\alpha$ to all $x \in X \setminus X_{\beta}$ such that there is $Y_x \subset X_{\beta}$ with $x = \sup Y_x$.  Then, for arbitrary 
$X$, taking $X' = \bigcup_{\alpha} X_{\alpha}$, this defines $\lambda_X \in \Iso(X')$.  We claim that since $X$ is Artinian and downward small, $X' = X$: if not, let $x$ be a minimal 
element of $X \setminus X'$.  Then $\lambda_X$ is defined on $D^{\circ}(x) = \{y \in X \ | \ y < x\}$ and $x = \sup D^{\circ}(x)$, so if 
$\alpha = \sup \lambda_X(D^{\circ}(x))$, then $\lambda_X(x) = \alpha$ (if $\alpha \notin \lambda_X(D^{\circ}(x))$ or $\lambda_X(x) = \alpha + 1$ (if $\alpha \in D^{\circ}(x)$).  
\end{proof}

\subsection{Length Functions}
\textbf{} \\ \\ \noindent
By Theorem \ref{0.0.2} and Lemma \ref{0.0.1}, for any
Artininan ordered set $X$, $\Iso(X)$ has a bottom element.  Indeed the map $\lambda_X: X \ra \Ord$ constructed in the proof of Theorem \ref{0.0.2} \emph{is} the bottom element of $\Iso(X)$.  Following \cite{Brookfield02}, we call $\lambda_X$ the \textbf{length function}. 
\\ \indent
If $X$ has a top element $T$, we define the \textbf{length of X} to be $\len(X) = \lambda_X(T)$.      
\\ \\
Example 1.1: For $\alpha \in \Ord$ and $x \leq \alpha$, $\lambda_{\alpha+1}(x) = x$, so $\len(\alpha +1) = \lambda_{\alpha+1}(\alpha) = \alpha$.  
\\ \\
Example 1.2: For $m,n \in \Z^+$, let $X_1 = \{0,\ldots,m\}$ and $X_2 = \{0,\ldots,n\}$, and let $X = X_1 \times X_2$.  Then for all $(i,j) \in X$, $\lambda_X(i,j) = i+j$ and $\len(X) = m+n$.    
\\ \\
For $\alpha, \beta \in \Ord$, we define the \textbf{Brookfield sum}
\[ \alpha \oplus_B \beta := \len( (\alpha + 1) \times (\beta + 1)). \]
We recall the following, a version of the \textbf{Cantor normal form}: for any $\alpha, \beta \in \Ord$ there are $\gamma_1,\ldots,\gamma_r \in \Ord$, $r \in \Z^+$ and $m_1,\dots,m_r,n_1,\ldots,n_r \in \N$ with
\[ \alpha = m_1 \omega^{\gamma_1} + \ldots + m_r \omega^{\gamma_r}, \ 
\beta = n_1 \omega^{\gamma_1} + \ldots + n_r \omega^{\gamma_r}. \]
This representation of the pair $(\alpha,\beta)$ is unique if we require $\max(m_i,n_i) > 0$ for all $i$.  We may then define the \textbf{Hessenberg sum} 
\[\alpha \oplus_H \beta = (m_1 + n_1) \omega^{\gamma_1} + \ldots + (m_r + n_r) \omega^{\gamma_r}. \]

\begin{thm}
\label{0.0.3}
For all $\alpha, \beta \in \Ord$, $\alpha \oplus_B \beta = \alpha \oplus_H \beta$.
\end{thm}
\begin{proof} See \cite[Thm. 2.12]{Brookfield02}.
\end{proof}
\noindent
In view of Theorem \ref{0.0.3} we write $\alpha \oplus \beta$ for 
$\alpha \oplus_B \beta = \alpha \oplus_H \beta$ and speak of the \textbf{Hessenberg-Brookfield sum}.  This operation is well-known to the 
initiates of ordinal arithmetic, who call it the ``natural sum''.  The 
next result collects facts about $\alpha \oplus \beta$ for our later use.
\begin{prop}
\label{HESSENBERGPROP}
Let $\alpha, \beta \in \Ord$.  \\
a) If $\beta < \omega$, then $\alpha + \beta = \alpha \oplus \beta$.  \\
b) In general we have 
\begin{equation}
\label{HESSENBERGINEQ}
\max (\alpha + \beta, \beta + \alpha) \leq \alpha \oplus \beta \leq 
\alpha \beta + \beta \alpha. 
\end{equation}
\end{prop}
\begin{proof} Left to the reader.
\end{proof}  

\section{Euclidean Functions}

\subsection{Basic Definitions}
\textbf{} \\ \\ \noindent
A \textbf{Euclidean function} is a function $\varphi: R^{\bullet} \ra \Ord$ such that for all $a \in R$, $b \in R^{\bullet}$, there are $q,r \in R$ with $a = qb + r$ and ($r = 0$ or $\varphi(r) < \varphi(b)$).  A ring is \textbf{Euclidean} if it admits a Euclidean function.
\\ \\
Example 2.1: Let $R = \Z$.  Then $n \mapsto |n|$ is a Euclidean function. 
\\ \\
Example 2.2: Let $k$ be a field and let $R = k[t]$.  Then 
$\varphi: R^{\bullet} \ra \Z^+$ given by $P \mapsto 1 + \deg P$ is a Euclidean function, as is $P \mapsto 2^{\deg P}$.  
\\ \\
Example 2.3: \cite[Prop. 5]{Samuel71} Let $R$ be a semilocal PID with nonassociate prime elements $\pi_1,\ldots,\pi_n$.  We may write $x \in R^{\times}$ as $u \pi_1^{a_1} \cdots 
\pi_n^{a_n}$ for $a_1,\ldots,a_n \in \N$, and then $x \mapsto a_1 + \ldots + a_n$ is a Euclidean function.


\begin{prop}
\label{1.2}
a) Let $\varphi: R^{\bullet} \ra \Ord$ be a Euclidean function.  For every nonzero ideal $I$ of $R$, let $x \in I^{\bullet}$ be such that $\varphi(x) \leq \varphi(y)$ for all $y \in I^{\bullet}$.  Then $I = \langle x \rangle$.  \\
b) In particular, a Euclidean ring is principal.  
\end{prop}
\begin{proof} Left to the reader; or see
\cite[Prop. 3]{Samuel71}.
\end{proof}
\noindent
\textbf{Extension at zero}: It will be convenient to define our 
Euclidean functions at the zero element of $R$.  There are several reasonable ways to do this.  Although the initially appealing one is to take $\varphi(0) = 0$ and require $\varphi$ to 
take nonzero values on $R^{\bullet}$, in the long run it turns out to be 
useful to take a quite different convention: we allow Euclidean functions to take the value zero at nonzero arguments -- so that, in particular, the bottom 
Euclidean function $\varphi_R$ will take the value $0$ precisely at the units -- 
and we define $\varphi(0) = \sup_{x \in R^{\bullet}} \varphi(x) + 1$.  This is actually not so strange: after all, $0$ is the \emph{top} element of $R$ with 
respect to the divisibility quasi-ordering.
\\ \indent

\subsection{Structure Theory of Principal Rings}

\begin{thm}
\label{ZSTHM} 
\label{0.4}
a) If $R = \prod_{i=1}^r R_i$, then $R$ is principal iff $R_i$ is principal for all $i$.  \\
b) For every principal ring $R$ there is $r \in \N$, a finite set of principal ideal domains $R_1,\ldots,R_n$ and a principal Artinian ring $A$ such that $R \cong \prod_{i=1}^n R_i \times A$.  The $R_i$'s are uniquely determined by $R$ up to isomorphism (and reordering), and $A$ is uniquely determined by $R$ up to isomorphism: we call $A$ the \textbf{Artinian part} of $R$.  \\ 
c) A ring is Artinian principal iff it is isomorphic to a finite product 
of \emph{local} Artinian principal rings. 
\end{thm}
\begin{proof} See \cite[p. 245]{Zariski-Samuel}.
\end{proof}

\subsection{Generalized Euclidean Functions}
\textbf{} \\ \\ \noindent
A \textbf{generalized Euclidean function} is a 
function $\varphi$ from $R^{\bullet}$ to an Artinian ordered class $X$ such that for all $x \in R, y \in R^{\bullet}$, there are 
$q,r \in R$ with $a = qx + r$ such that either $r = 0$ or $\varphi(r) < \varphi(b)$.
\\ \\
As several authors have observed over the years, nothing in the theory of 
Euclidean functions is lost by entertaining Euclidean functions with values 
in any Artinian ordered class.  In particular, the proof of 
Proposition \ref{1.2} carries over easily to show that a ring admitting a 
generalized Euclidean function is principal.  But there is a better way 
to see this:
\begin{lemma}
\label{0.0.4}
Let $X,Y$ be Artinian ordered classes, $\varphi: R \ra X$ a generalized Euclidean function and $f: X \ra Y$ an isotone map.  Then $f \circ \varphi: R \ra Y$ is a generalized Euclidean function.
\end{lemma}
\begin{proof}
Left to the reader.
\end{proof}

\begin{cor}(\cite[Prop. 11]{Samuel71}, \cite[Prop. 4]{Nagata85})
\label{0.0.5}
A ring which admits a generalized Euclidean function is Euclidean.  
\end{cor}
\begin{proof}
If $\varphi: R \ra X$ is generalized Euclidean,  $\lambda_X \circ \varphi: R \ra \Ord$ is Euclidean.
\end{proof}
\noindent
Proposition \ref{0.0.5} may suggest that there is nothing to gain in considering 
Euclidean functions with values in a non-well-ordered sets.  But this is not the case! 

\begin{lemma}(\cite[Thm. 2]{Nagata85})
\label{NAGATALEMMA}
Let $R_1,R_2$ be commutative rings, $X_1,X_2$ be Artinian ordered classes, and $\varphi_1: R_1 \ra X_1$, $\varphi_2: R_2 \ra X_2$ be generalized Euclidean 
functions.  Then $\varphi_1 \times \varphi_2: R_1 \times R_2 \ra X_1 \times X_2$ 
is a generalized Euclidean function.
\end{lemma}
\begin{proof} Let $x = (x_1,x_2) \in R$, $y = (y_1,y_2) \in R^{\bullet}$.  
We may assume $y \nmid x$ and thus $x \in R^{\bullet}$.  Since $\varphi_1$ and $\varphi_2$ are Euclidean, for $i = 1,2$ there are $q_i,r_i \in R_i$ with $x_i = q_i y_i + r_i$ and ($r_i = 0$ or $\varphi_i(r_i) < \varphi_i(y_i)$).  Since $y \nmid x$, $r \neq 0$. Now: \\
Case 1: Suppose that any one of the following occurs: \\
(i) $r_1 \neq 0$, $r_2 \neq 0$.  \\
(ii) $r_1 = y_1 = 0$, and thus $r_2,y_2 \neq 0$.  \\
(iii) $r_2 = y_2 = 0$, and thus $r_1,y_1 \neq 0$.  \\
Put $q = (q_1,q_2)$ and $r = (r_1,r_2)$, so $x = qy + r$ and 
\[\varphi(r) = (\varphi(r_1),\varphi(r_2)) < (\varphi(y_1),\varphi(y_2)) = \varphi(y). \]  
Case 2: Suppose $r_1 = 0$, $y_1 \neq 0$, and thus $r_2 \neq 0$.  Then 
take $q = (q_1-1,q_2)$ and $r = (y_1,r_2)$, so $x = qy + r$ and 
\[ \varphi(r) = (\varphi(y_1),\varphi(r_2)) < (\varphi(y_1),\varphi(y_2)) = 
\varphi(y). \]
Similarly if $r_2 = 0$, $y_2 \neq 0$.  \\
Case 3: Suppose $r_2 =0$, $y_2 \neq 0$, and thus $r_1 \neq 0$.  Put 
$q = (q_1,q_2-1)$ and $r = (r_1,y_2)$, so $x = qy +r$ and 
\[ \varphi(r) = (\varphi(r_1),\varphi(y_2)) < (\varphi(y_1),\varphi(y_2)) = 
\varphi(y). \]
\end{proof}

\subsection{Isotone Euclidean Functions}
\textbf{} \\ \\ \noindent
For a ring $R$, the divisibility relation is a \textbf{quasi-ordering} -- i.e., 
reflexive and transitive but not necessarily anti-symmetric.  If $X$ and $Y$ 
are quasi-ordered sets, we can define an isotone map $f: X \ra Y$ just as above:  
 if $x_1 < x_2 \implies f(x_1) < f(x_2)$.  
\begin{lemma}
Let $X$ be a quasi-ordered set and $Y$ be an ordered set.  Suppose 
that $f: X \ra Y$ is an isotone map.  Then $X$ is ordered.
\end{lemma}
\begin{proof}
For $x_1, x_2 \in X$, $x_1 < x_2, \ x_2 < x_1 \implies f(x_1) < f(x_2), \ f(x_2) < f(x_1)$.
\end{proof}
\noindent
If $X$ is a quasi-ordered set, it has an ordered completion: i.e., an ordered 
set $\overline{X}$ and a weakly isotone map $X \ra \overline{X}$ which is 
universal for weakly isotone maps from $X$ into an ordered set.  Indeed, we 
simply take the quotient of $X$ under the equivalence relation $x_1 \sim x_2$ 
if $x_1 \leq x_2$ and $x_2 \leq x_1$.  If we do this for the divisibility relation on $R$, we get precisely the ordered set $\Prin R$ of principal ideals 
of $R$.  All this is to motivate the following definition.  \\ \indent
A Euclidean function $\varphi: R \ra \Ord$ is \textbf{weakly isotone} (resp. \textbf{isotone}) if whenever $x$ divides $y$, $\varphi(x) \leq \varphi(y)$ (resp. whenever $x$ strictly divides $y$, $\varphi(x) < \varphi(y)$.  
\\ \\
Example 2.4: Define $\varphi: \Z^{\bullet} \ra \Ord^{\bullet}$ by $1 \mapsto 2$ 
and $n \mapsto |n|$ else.  Then $\varphi$ is Euclidean.  But $(1) = (-1)$ and $\varphi(1) \neq \varphi(-1)$, so $\varphi$ 
is not weakly isotone.  

\begin{thm}
\label{1.4}
\label{SAMUEL1}
Let $\varphi: R^{\bullet} \ra \Ord$ be a Euclidean function.  Then the set of isotone Euclidean functions $\psi \leq \varphi$ has a maximal element 
\[\underphi: x \in R^{\bullet} \mapsto \min_{y \in (x)^{\bullet}} \varphi(y). \]
\end{thm}
\begin{proof}
Step 1: We show $\underphi$ is a Euclidean function: let $a \in R, \ b \in R^{\bullet}$.  Then there exists $c \in R$ such that $bc \neq 0$ and $\underphi(b) = \varphi(bc)$.  
Since $\varphi$ is Euclidean, there are $q,r \in R$ with $a = qbc + r$ and 
either $r = 0$ -- so $b \ | a $--  or $\underphi(r) \leq \varphi(r) < \varphi(bc) = \underphi(b)$.  \\
Step 2: We show $\underphi$ is isotone: If $ac \neq 0$, then $\underphi(a) = \min_{y \ | \ ay \neq 0} \varphi(ay) 
\leq \min_{cy \ | \ acy \neq 0} \varphi(acy) = \underphi(ac)$.  Conversely, 
suppose $\underphi(ac) = \underphi(a)$, and write $a = qac + r$ with $r = 0$ 
or $\underphi(r) < \underphi(ac) = \underphi(a)$.  But if $r \neq 0$, then by 
what we have just shown, $\underphi(r) = \underphi(a(1-qc)) \geq \underphi(a)$, 
a contradiction.  So $r = 0$ and $(a) = (ac)$. \\
Step 3: By construction $\underphi \leq \varphi$.  Moreover, if $\psi \leq \varphi$ is an isotone Euclidean function, then for all $a,c \in R$ with $ac \neq 0$, $\psi(a) \leq \psi(ac) \leq \varphi(ac)$, so $\psi(a) \leq \underphi(a)$.  
\end{proof}

\begin{cor}
\label{SAMUEL1COR}
A Euclidean function is weakly isotone iff it is isotone.  
\end{cor}
\begin{proof} Of course any isotone function is weakly isotone.  Conversely, let $\varphi: R^{\bullet} \ra \Ord$ be a weakly isotone Euclidean function, let 
$a,c \in R$ with $ac \neq 0$, and suppose $\varphi(ac) = \varphi(c)$.  Write 
$a = qac + r$ with $r = 0$ or $\varphi(r) < \varphi(ac) = \varphi(a)$.  If 
$r \neq 0$, then $\varphi(r) = \varphi(a(1-qc)) \geq \varphi(a)$, contradiction.  
So $r = 0$ and $(a) = (ac)$.  
\end{proof}

\subsection{The Bottom Euclidean Function and the Euclidean Order Type}
\textbf{} \\ \\ \noindent
For any commutative ring $R$, let $\Euc(R) \subset \Ord^R$ be the subclass of Euclidean functions $\varphi: R \ra \Ord$, with the induced partial ordering.

\begin{lemma}
\label{0.0.1.5}
Let $R$ be a commutative ring.  Then every nonempty subclass of $\Euc(R)$ has an infimum 
in $\Euc(R)$.  
\end{lemma}
\begin{proof}
Let $\mathcal{C} = \{\varphi_c\}$ be a nonempty subclass of $\Euc(R)$, and let $\varphi$ be the infimum in $\Ord^R$; it suffices to show that 
$\varphi$ is Euclidean.  Let $a,b \in R$ with $b \nmid a$.  Choose $i \in I$ such that $\varphi(b) = \varphi_i(b)$.  Since $\varphi_i$ is Euclidean, there are $q,r \in R$ such that $\varphi_i(r) < 
\varphi_i(b)$, and then $\varphi(r) \leq \varphi_i(r) < \varphi_i(b) = \varphi(b)$.
\end{proof}

\begin{thm}
\label{BOTTOMTHM}
Let $R$ be a Euclidean ring.  \\
a) The class $\Euc(R)$ of all Euclidean functions on $R$ has a bottom element $\varphi_R$. \\
b) The bottom Euclidean function $\varphi_R$ is isotone.  \\
c) The set $\varphi_R(R)$ is an ordinal.
\end{thm}
\begin{proof}
a) This is immediate from Lemma \ref{0.0.1.5}.  \\
b) Since $\varphi_R$ is the bottom Euclidean function, we must 
have $\varphi_R = \underline{\varphi_R}$.  \\
c) It's enough to show $\varphi_R(R)$ is downward closed: we  
need to rule out the existence of $\alpha < \beta \in \Ord$ such that $\varphi_R(R)$ contains $\beta$ but not $\alpha$.  Let $\alpha'$ be the least 
element of $\varphi_R(R)$ exceeding $\alpha$.  If we redefine $\varphi_R$ to take the value $\alpha$ whenever $\varphi_R$ 
takes the value $\alpha'$, we get a smaller Euclidean function than $\varphi_R$: contradiction.
\end{proof}

\begin{prop}
\label{BOTTOMISOPROP}
If $R$ is Euclidean, the bottom Euclidean function $\varphi_R$ 
is isotone.
\end{prop}
\begin{proof}
By Theorem \ref{1.4}, $\underline{\varphi_R}$ is an isotone Euclidean function.  But $\varphi_R$ is the bottom Euclidean function, so $\underline{\varphi_R} = \varphi_R$.  
\end{proof}
\noindent
For any Euclidean ring $R$ we define the 
\textbf{order type} $e(R) = \varphi_R(R) \in \Ord$.  This ordinal invariant of $R$ is our main object of interest: given a Euclidean ring $R$ 
we would like to compute its order type $e(R)$; conversely we would like 
to know which ordinals arise as order types of Euclidean rings. 
\\ \\
Example 2.5: Let $R = \Z$.   The map $n \in \Z \mapsto \#$ of binary digits of $|n|$ is the bottom Euclidean function, so $e(\Z) = \omega$.  
\\ \\
Example 2.6: Let $R = k[t]$.  The map $x \in k[t] \mapsto 1 + \deg t$ is the bottom Euclidean function, so $e(k[t])) = \omega$.

\begin{thm}
\label{NOZVALUEDTHM}
\label{1.9}
Suppose $R$ is Euclidean with $e(R) \leq \omega$.  Then: \\
a) For all $x \in R^{\bullet}$, $R/(x)$ is Artinian. \\
b) $R$ is either a PID or an Artinian principal ring.
\end{thm}
\begin{proof} 
a) If $R/(x)$ were not Artinian, there would be a sequence of elements $\{x_n\}_{n=0}^{\infty}$ in $R$ with $x_0 = x$ and $(x_{n+1}) \subsetneq (x_n)$ for all $n \in \N$.  Applying Corollary \ref{SAMUEL1COR}b)  
we get $\varphi(x) = \varphi(x_0) > \varphi(x_1) > \ldots > \varphi(x_n) > \ldots$, 
contradicting  $\varphi(x) \in \omega$.  \\
b) This follows immediately from Theorem \ref{ZSTHM}.
\end{proof}

\begin{cor}(Fletcher \cite{Fletcher71}) \\
\label{1.10}
a) If $A$ is an Artinian Euclidean ring, then $e(R) < \omega$.  \\
b) A Euclidean ring $R$ with $e(R) = \omega$ is a domain.
\end{cor}
\begin{proof}
a) The value of $\varphi_A$ at $x \in R$ depends only 
on the ideal $(x)$.  But an Artinian principal ring has only finitely many ideals!  So 
$e(R) < \omega$.  \\
b) This follows immediately from Theorem \ref{1.9}b) and part a).
\end{proof}

\subsection{The Localized Euclidean Function}

\begin{thm}
\label{9.9}
Let $R$ be a Euclidean domain, and let $S \subset R$ be a multiplicatively closed subset.  Then the localization $S^{-1} R$ is Euclidean and $e(S^{-1} R) \leq e(R)$.
\end{thm}
\begin{proof} See \cite[$\S 4$]{Motzkin49} or Samuel \cite[Prop. 7]{Samuel71}.
\end{proof}


\subsection{The Quotient Euclidean Function}

\begin{thm}
\label{QUOTIENTPROP}
\label{QUOTIENTTHM}
\label{1.11}
Let $\varphi: R \ra \Ord$ be a Euclidean function.  Let $b \in R^{\bullet}$ be an ideal of $R$, 
and let $f: R \ra R/(b)$ be the quotient map.  For $x \in R/(b)$, let $\tilde{x} \in f^{-1}(x)$ be any element such that $\varphi(\tilde{x}) \leq \varphi(y)$ for all 
$y \in f^{-1}(x)$. \\
a)  Then $\varphi': x \in (R/(b))^{\bullet} \mapsto \varphi(\tilde{x})$ is a Euclidean function.  \\
b) For the bottom Euclidean function $\varphi_R$, we have 
\begin{equation}
\label{1.11EQ}
\varphi_R'(0) = \sup_{\overline{x} \neq \overline{0}} \varphi_R'(\overline{x}) +1 = \varphi_R(b). 
\end{equation}    
\end{thm}
\begin{proof}
a) For $x \in R/(b)$, $y \in (R/(b))^{\bullet}$, $\tilde{x} \in R$, $\tilde{y} \in R^{\bullet}$, so there are $q,r \in R$ with $\tilde{x} = q \tilde{y} + r$ and $\varphi(r) < \varphi(\tilde{y})$.  Then $x = f(q)y + f(r)$ and hence \[\varphi'(f(r)) \leq \varphi(r) < \varphi(\tilde{y}) = \varphi'(y). \]
b) The first equality in (\ref{1.11EQ}) is the definition of the extension to 
$0$ of any Euclidean function.  The second equality is a key -- in fact, characteristic -- property of the bottom Euclidean function which is proved in \cite[Prop. 10]{Samuel71}.
\end{proof}

\begin{cor}
\label{1.12}
If $R$ is Euclidean, so is every quotient ring $R'$, and $e(R') \leq e(R)$.
\end{cor}

\subsection{The Product Theorem}

\begin{lemma}(Ordinal Subtraction)
\label{SUBTRACTLEMMA} \\
a) For $\alpha \leq \beta \in \Ord$, there is a unique $\gamma \in \Ord$ such that $\alpha + \gamma = \beta$.  We may therefore define 
\[ -\alpha + \beta = \gamma. \]
b) Suppose we have ordinals $\alpha,\beta,\gamma$ such that $\gamma \leq \alpha < \beta$.  Then $-\gamma + \alpha < -\gamma + \beta$.  
\end{lemma}
\begin{proof} 
a) Existence of $\gamma$: If $\alpha = \beta$, then we take $\gamma = 0$.  Otherwise, 
$\alpha \subsetneq \beta$; let $x_0$ be the least element of $\beta \setminus \alpha$ and let $\gamma$ be the order type of $\{x \in \beta \ | \ x \geq x_0\}$. \\ 
Uniqueness of $\gamma$: suppose we have two well-ordered sets $W_1$ and $W_2$ 
such that $\alpha + W_1$ is order-isomorphic to $\alpha + W_2$.  Then the unique 
order-isomorphism between them induces an order-isomorphism from $W_1$ to $W_2$.  \\
b) For if not, $-\gamma + \beta \leq -\gamma + \alpha$, and then $\beta = \gamma + (-\gamma + \beta) \leq \gamma + (-\gamma + \alpha) = \alpha$.
\end{proof}

\begin{thm}(Product Theorem)
\label{PRODUCTTHM}
Let $R_1,\ldots,R_n$ be Euclidean rings.\\
a) The ring $\prod_{i=1}^n R_i$ is Euclidean iff $R_i$ is Euclidean for all $i$.  \\
b) If the equivalent conditions of part a) hold, then 
\begin{equation}
\label{PRODTHMEQ}
e(R_1) + \ldots + e(R_n) \leq e(\prod_{i=1}^n R_i) \leq e(R_1) \oplus \ldots 
\oplus e(R_n). 
\end{equation}
\end{thm}
\begin{proof}
Induction reduces us to the case $n = 2$.  Put 
$R = R_1 \times R_2$.  \\
a) This is immediate from Lemma \ref{NAGATALEMMA}, Corollary \ref{0.0.5} and 
Theorem \ref{QUOTIENTTHM}.  \\
b) Let $b = (0,1)$, so $R/(b) = R_1$ and thus by Theorem \ref{QUOTIENTTHM}b), we 
have $\varphi_R(b) = e(R_1)$.  For $y \in R_2$, $b = (0,1) \ | \ (0,y)$, 
so by Proposition \ref{BOTTOMISOPROP}, $\varphi_R((0,1)) \leq \varphi_R((0,y))$.  
By Lemma \ref{SUBTRACTLEMMA} we may put 
\[ \psi(y) = - \varphi_R((0,1)) + \varphi_R((0,y)). \]
We {\sc claim} $\psi: R_2 \ra \Ord$ is a Euclidean function.  Granting this 
for the moment, it then follows that $\psi \geq \varphi_{R_2}$, so 
\[ e(R) = \varphi_R((0,0)) = \varphi_R((0,1)) + \psi(0) \geq e(R_1) + e(R_2). \]
{\sc proof of claim}: Let $x \in R_2, \ y \in R_2^{\bullet}$; as usual, we may assume $y \nmid x$.  Since $\varphi_R$ is Euclidean, there are $q = (q_1,q_2), \ r = (r_1,r_2) \in R$ such that 
$(0,x) = q (0,y) + r = (r_1,q_2 y + r_2)$ and either $r = 0$ or 
$\varphi_R(r) < \varphi_R((0,y))$.  Thus $r_1 = 0$ and $x = q_2 y + r_2$.  Since 
$y \nmid x$ we have $r_2 \neq 0$, so $r \neq 0$ and thus $\varphi_R((0,r_2)) < 
\varphi_R((0,y))$.  By Lemma \ref{SUBTRACTLEMMA}b) we may subtract $\varphi_R(0,1)$ -- on the left! -- from both sides to get 
\[ \psi(r_2) = -\varphi_R((0,l)) - \varphi_R((0,r_2)) < - \varphi_R((0,1)) - 
\varphi_R((0,y)) = \psi(y). \]
For the second inequality of (\ref{PRODTHMEQ}), let $\varphi_1: R_1 \ra e(R_1)$, $\varphi_2: R_2 \ra e(R_2)$ be the bottom 
Euclidean functions on $R_1$ and $R_2$.  By definition of the Hessenberg-Brookfield sum, we have an isotone map $\lambda: (e(R_1) + 1) \times (e(R_2) + 1) \ra e(R_1) \oplus e(R_2)$ and thus a Euclidean function 
$\lambda \circ (\varphi_1 \times \varphi_2): R \ra e(R_1) \oplus e(R_2)$. 
\end{proof} 
\noindent
Remark 2.7: The upper bound on the order type of the product ring in (\ref{PRODTHMEQ}) is essentially due to Nagata.  Moreover, in the proof of Proposition 6 of \cite{Samuel71}, Samuel gives the bound 
$e(R_1 \times R_2) \leq e(R_1) \times e(R_2) + e(R_2) \times e(R_1)$.  By Proposition \ref{HESSENBERGPROP}b), Samuel's bound is not as good as Nagata's: 
there is a penalty to pay for employing the ``usual'' ordinal operations rather than the Hessenberg-Brookfield sum.

\subsection{Applications}
\textbf{} \\ \\ \noindent
Let $R$ be a nonzero local principal ring with maximal ideal $(\pi)$.  Then every element $x \in R^{\bullet}$ may be written as $u \pi^a$ for
$u \in R^{\times}$ and a unique $a \in \N$.  If $R$ is a domain, then we take the convention that $0 = 1 \cdot \pi^{\omega}$.  Otherwise $R$ is Artinian 
and there is a least positive integer $a$ such that $0 = 1 \cdot \pi^a$; 
this $a$ is nothing else than $\ell(R)$, the \textbf{length} of $R$ as an $R$-module.  

\begin{thm}
\label{7.1}
a) (Samuel) Let $R$ be a local principal ring with maximal ideal generated by $\pi$.  
The map $\varphi: R \ra \Ord$ by $u \pi^a \mapsto a$ is a Euclidean function on $R$.  \\
b) In fact $\varphi = \varphi_R$ is the bottom Euclidean function, hence: \\
$\bullet$ If $R$ is a domain, $e(R) = \omega$.  \\
$\bullet$ If $R$ is Artinian, then $e(R) = \ell(R)$.
\end{thm}
\begin{proof}
a) Let $x = u_1 \pi^a \in R$, $y = u_2 \pi^b \in R^{\bullet}$.  If $b \leq a$ then $y \ | \ x$ and we 
may write $x = qy + 0$.  If $b > a$ we may write $x = 0 \cdot y + x$, and then 
$\varphi(x) < \varphi(y)$.  \\
b) Let $\varphi_R$ be the bottom Euclidean function on $R$, which is isotone by Proposition \ref{BOTTOMISOPROP}. Suppose first that $R$ is a domain.  Then for all $a \in \omega$ we must have 
 \[ 0 = \varphi_R( \pi^0) < \varphi_R(\pi^1) < \ldots < \varphi_R(\pi^a) < \ldots, \]
and thus $\varphi \leq \varphi_R$, so $\varphi_R = \varphi$ and \[e(R) = 
\varphi(0) = \sup_{a \in \N} \varphi(\pi^a) + 1 = \omega. \]
Similarly, if $R$ is Artinian then 
\[ 0 = \varphi_R(\pi^0) < \varphi_R(\pi^1) < \ldots < \varphi_R(\pi^{\ell(R)}), \]
and as above this forces $\varphi_R = \varphi$ and $e(R) = \varphi(0) = \ell(R)$.
\end{proof}

\begin{cor}
\label{FLETCHERCOR1}
If $R$ is an Artinian principal ring, then $R$ is Euclidean and $e(R) = \ell(R)$, the length of $R$ as an $R$-module.  
\end{cor}
\begin{proof} This follows immediately from Theorem \ref{7.1} and Theorem \ref{PRODUCTTHM}.
\end{proof}

\begin{thm}
\label{FLETCHERTHM}
Let $R = \prod_{i=1}^r R_i \times A = R' \times A$ be a Euclidean ring.  \\
a) We have $e(R) = e(R') + \ell(A)$.  \\
b) $e(R')$ is a limit ordinal (possibly zero).  \\
c) $e(R') \geq r \omega$.
\end{thm}
\begin{proof}
a) By Corollary \ref{FLETCHERCOR1}, $A$ is Euclidean with $e(A) = \ell(A)$;  then by Theorem 
\ref{PRODUCTTHM}, \[e(R') + \ell(A) = e(R') + e(A) \leq e(R) \leq e(R') \oplus e(A) = 
e(R') + e(A) = e(R') + \ell(A). \]
The remaining assertions hold trivially if $r = 0$, so we assume $r \geq 1$.  \\
b) Since $R'$ is a product of domains, the set of nonzero ideals of $R$ has no minimal element, so $e(R')$ is a nonzero limit ordinal.  \\
c) By part b) and Theorem \ref{PRODUCTTHM}, $r \omega \leq e(R_1) + \ldots + e(R_r) \leq e(R')$.
\end{proof}

\begin{cor}(Fletcher)
\label{FLETCHERCOR2}
A Euclidean ring $R$ with $e(R) = \omega$ is a domain.
\end{cor}
\noindent
We say a Euclidean ring $R \cong \prod_{i=1}^r R_i \times A(R)$ is \textbf{small} if $e(R_i) = \omega$ for all $i$; otherwise we say $R$ is \textbf{large}.  

\begin{thm}
\label{SMALLSTRUCTURETHM}
a) Let $R \cong \prod_{i=1}^r R_i \times A$ be a small Euclidean ring.  Then 
\[ e(R) = r \omega + \ell(A), \]
where $n$ is the \emph{length} of the Artinian principal ring $A$.  \\
b) For every ordinal $\alpha < \omega^2$, there is a small Euclidean ring 
$R$ with $e(R) = \alpha$.   
\end{thm}
\begin{proof}
a) By Theorem \ref{FLETCHERTHM}, $e(A) = \ell(A) < \omega$.  Now we apply the 
Product Theorem: 
\[ r \omega + \ell(A) = e_1(R) + \ldots + e_r(R) + e(A) \leq e(R) \leq  \bigoplus_{i=1}^r e(R_i) \oplus \ell(A) = r \omega + \ell(A). \]
b) The ordinals less than $\omega^2$ are of the form $r \omega + n$ for $r,n \in \omega$.  By part a),
\[e(\C[t]^r \times \C[t]/(t^n)) = r \omega + n. \]
\end{proof}
\noindent
Are all Euclidean rings small?  If so, Theorem \ref{SMALLSTRUCTURETHM} would be the ultimate result on Euclidean order types.  This question was implicit in \cite{Motzkin49} and made explicit in 
\cite{Samuel71}.  It was later answered \emph{negatively} by Hiblot \cite{Hiblot75}, \cite{Hiblot77} and 
Nagata \cite{Nagata78}.  It seems that the Euclidean order type of these large Euclidean domains has never been investigated...and, alas, will not investigated here.  However, the following result shows that 
Euclidean rings of the sort familiar in number theory are small.

\begin{prop}(\cite[Prop. 15]{Samuel71})
Let $R$ be a Euclidean domain such that $R/(a)$ is a finite ring for all $a \in R^{\bullet}$.  Then $R$ is small.
\end{prop}
\begin{proof}
If not, there is $b \in R$ with $\varphi_R(b) = \omega$.  Let $\overphi: R/(b) \ra \Ord$ be the quotient Euclidean function.  Write the elements of $R/(b)$ 
as $\overline{x_1} = 0, \overline{x_2},\ldots, \overline{x_n}$.  For all $i > 1$, $\overphi( \overline{x_i}) < \overphi( \overline{0}) \leq \omega$.  But 
$\overphi(\overline{0}) = \sup_{i > 1} \overphi( \overline{x_i}) + 1 < \omega$.  
Thus, there exists $b' \in (b)$ with $\varphi_R(b') < \omega$, contradicting 
the fact that $\varphi_R$ is isotone.
\end{proof}
\noindent

\section{Length Functions on Rings}

\subsection{The length function on a Noetherian ring}
\textbf{} \\ \\ \noindent
In this section we closely follow work of Gulliksen \cite{Gulliksen73} and Brookfield \cite{Brookfield02}.
\\ \\
Let $R$ be a ring, and let $\mathcal{I}(R)$ be the lattice of ideals of $R$.  
Then $\mathcal{I}(R)$ is Noetherian (resp. Artinian) iff $R$ is Noetherian 
(resp. Artinian).  Thus the dual lattice $\mathcal{I}^{\vee}(R)$ is Artinian (resp. Noetherian) iff $R$ is Noetherian (resp. Artinian).  \\ \indent
Henceforth we suppose $R$ is Noetherian, so $\mathcal{I}^{\vee}(R)$ 
is Artinian with top element $(0)$.  By the results of $\S 1$ there is a least isotone map $\lambda_R: \mathcal{I}^{\vee}(R) \ra \Ord$, 
the \textbf{length function} $\lambda_R$ of $R$, and we define the 
\textbf{length of R} as $\len(R) = \lambda_R((0))$.  
\\ \\
For any ideal $I$ of $R$, $R/I$ is Noetherian, so $\lambda_{R/I}$ and 
$\len(R/I)$ are well-defined.  If we denote the quotient map $R \ra R/I$ by 
$q$, then the usual pullback of ideals $q^*$ identifies $\mathcal{I}(R/I)$ 
with an ordered subset of $\mathcal{I}(R)$ and hence also $\mathcal{I}^{\vee}(R/I)$ with an ordered subset of $\mathcal{I}^{\vee}(R)$, 
and it is easy to see that under this identification we have 
\[ \lambda_R|_{\mathcal{I}^{\vee}(R/I)} = \lambda_{R/I}, \]
and thus also 
\[ \len(R/I) = \lambda_{R/I}((0))) = \lambda_R (q^*((0))) = \lambda_R(I). \]
\\ 
To ease notation, for $x \in R$ we put $\ell(x) = \lambda_R( (x))$.  

\begin{prop}
\label{3.1}
Let $\varphi$ be a Euclidean function on $R$.  For all $x \in R$, $\ell(x) \leq \varphi(x)$.  
\end{prop}
\begin{proof} Since $R$ admits a Euclidean function, it is a principal ring, 
and thus $\mathcal{I}^{\vee}(R) = R^{\bullet}/R^{\times}$.  We may assume 
that $\varphi = \varphi_R$ is the bottom Euclidean function on $R$.  Then 
both $\ell$ and $\varphi$ induce well-defined isotone functions on $R^{\bullet}/R^{\times}$.  But by definition $\ell = \lambda_R$ is the 
\emph{least} isotone function on $\mathcal{I}^{\vee}(R)$, so $\ell(x) \leq 
\varphi_R(x)$ for all $x \in R$.
\end{proof}
\noindent
If $R$ is a PID and $x \in R^{\bullet}$, then the ring $R/(x)$ is an Artinian ring and thus its length, which is equal to $\ell(x)$, is finite.  In particular, for all $x \in R^{\bullet}$ $\ell(x) < \omega$ and $\ell(0) = \omega$.  From this the next result follows directly.

\begin{prop}
\label{3.2}
Let $R$ be a PID which is not a field, and let $x \in R$.  \\
a) If $x \in R^{\times}$, then $\ell(x) = 0$.  \\
b) If $x \in R^{\bullet} \setminus R^{\times}$, we may write $x = \pi_1 \cdots 
\pi_n$ for not necessarily distinct prime elements $\pi_1,\ldots,\pi_n$, and 
then $\ell(x) = n$.  \\
c) We have $\ell(0) = \len(R) = \omega$.
\end{prop}

\subsection{$\ell$-Euclidean rings}
\textbf{} \\ \\
A ring $R$ is $\ell$-\textbf{Euclidean} if the function $x \in R \mapsto 
\ell(x) \in \Ord^{\bullet}$ is a Euclidean function on $R$.  The point is that if $\ell$ is a Euclidean function on $R$, it is then 
the least Euclidean function on $R$, so that $e(R) = \len(R)$.  
\\ \\
Example 3.1: a) The ring $\Z$ is \emph{not} $\ell$-Euclidean. \\
b) For a field $k$, the ring $k[t]$ is $\ell$-Euclidean iff $k$ is algebraically closed.  
\\ \\
Example 3.2: A \textbf{norm} on a nonzero ring $R$ is a function 
$|\cdot| R \ra \N$ such that $|x| = 0 \iff x = 0$, $|x| = 1 \iff x \in R^{\times}$ and $|xy| = |x| |y|$ for all $x,y \in R$.  A ring admitting a norm is necessarily a domain: if $x,y \in R^{\bullet}$, 
$|xy| = |x| |y| \neq 0$, so $xy \in R^{\bullet}$.  An $\ell$-Euclidean domain admits a Euclidean norm: $x \in R^{\bullet} \mapsto 2^x$.

\begin{prop}
\label{3.5}
A localization of an $\ell$-Euclidean domain is $\ell$-Euclidean.
\end{prop}
\begin{proof} Left to the reader.
\end{proof}
\noindent
Proposition \ref{3.5} furnishes further examples of 
$\ell$-Euclidean domains, namely any ring between $\overline{k}[t]$ and 
its fraction field.  One wonders about further examples.

\begin{ques}
Is there a classification of $\ell$-Euclidean domains?
\end{ques}




\begin{thebibliography}{AF95}



\bibitem[Br02]{Brookfield02} G. Brookfield, \emph{The length of Noetherian modules}. Comm. Algebra 30 (2002), 3177–-3204.

\bibitem[Fl71]{Fletcher71} C.R. Fletcher, \emph{Euclidean rings}. J. London Math. Soc. 4 (1971), 79-–82. 

\bibitem[Gu73]{Gulliksen73} T.H. Gulliksen, \emph{A theory of length for Noetherian modules}. J. Pure Appl. Algebra 3 (1973), 159–-170.


\bibitem[He06]{Hessenberg} G. Hessenberg, \emph{Grundbegriffe der Mengenlehre}. G\"ottingen, 1906.

\bibitem[Hi75]{Hiblot75} J.-J. Hiblot, \emph{Des anneaux euclidiens dont le plus petit algorithme n'est pas \`a valeurs finies}. 
C. R. Acad. Sci. Paris S\'er. A-B 281 (1975), no. 12, Ai, A411–-A414. 

\bibitem[Hi77]{Hiblot77} J.-J. Hiblot, \emph{Correction \`a une note sur les anneaux euclidiens: ``Des anneaux euclidiens dont le plus petit algorithme n'est pas \`a valeurs finies'' (C. R. Acad. Sci. Paris S\'er. A-B 281 (1975), no. 12, A411--A414)}. C. R. Acad. Sci. Paris S\;er. A-B 284 (1977), no. 15, A847–-A849. 




\bibitem[Mo49]{Motzkin49} T. Motzkin, \emph{The Euclidean algorithm}.
Bull. Amer. Math. Soc. 55 (1949), 1142–-1146.

\bibitem[Na78]{Nagata78} M. Nagata, \emph{On Euclid algorithm. C. P. Ramanujam—a tribute}, pp. 175--186, Tata Inst. Fund. Res. Studies in Math., 8, Springer, Berlin-New York, 1978. 

\bibitem[Na85]{Nagata85} M. Nagata, \emph{Some remarks on Euclid rings}. J. Math. Kyoto Univ. 25 (1985), 421-–422. 


\bibitem[Sa71]{Samuel71} P. Samuel, \emph{About Euclidean rings}. J. Algebra 19 (1971), 282-–301.

\bibitem[ZS]{Zariski-Samuel} O. Zariski and P. Samuel, \emph{Commutative algebra, Volume I.} Van Nostrand, Princeton, New Jersey, 1958.

\end{thebibliography}
\end{document}